\documentclass[12pt, oneside]{scrartcl}
\usepackage[english]{babel}
\usepackage[T1]{fontenc}
\usepackage{amssymb}
\usepackage{amsmath}
\usepackage{hhline}
\usepackage{longtable}
\usepackage{amscd}
\usepackage{theorem}
\usepackage{array}
\usepackage{delarray}
\usepackage{multicol}
\usepackage{makeidx}

\usepackage{float}

\usepackage{tikz}
\usepackage{calc}
\usetikzlibrary{angles,quotes}
\usetikzlibrary{calc}
\usetikzlibrary{patterns}
\usetikzlibrary{graphs}
\usetikzlibrary{graphs.standard}
\usetikzlibrary{decorations.markings,shapes.geometric,positioning}

\newtheorem{theorem}{Theorem}%[section]

\newtheorem{corollary}{Corollary}%[section]
\newtheorem{remark}{Remark}%[section]
\newtheorem{claim}{Claim}[theorem]

\newtheorem{problem}{Problem}
\newenvironment{proof}[1][Proof]{\textbf{#1.} }{\ \rule{0.5em}{0.5em}}

\usepackage{tikz}
\author{Ayman El Zein\footnote{Computer Science Department, University of Sciences and Arts in Lebanon, Beirut, Lebanon; KALMA Laboratory, Faculty of Sciences, Lebanese University, Beirut, Lebanon}, Maidoun Mortada\footnote{KALMA Laboratory, Faculty of Sciences, Lebanese University, Baalbek, Lebanon; Basic and Applied Sciences Research, Al Maaref University, Beirut, Lebanon; and Graph Theory and Operation Research, Department of Mathematics and Physics, Lebanese International University (LIU), Beirut, Lebanon}}
    
\begin{document} 
\title{Advances on the Packing Coloring Conjectures of Subcubic Graphs}
\maketitle

\begin{abstract}
For a non-decreasing sequence of integers $S=(s_1,s_2, \dots, s_k)$, an $S$-packing coloring of $G$ is a partition of $V(G)$ into $k$ subsets $V_1,V_2,\dots,V_k$ such that the distance between any two distinct vertices $x,y \in V_i$ is at least $s_{i}+1$, $1\leq i\leq k$. The packing chromatic number $\chi_{\rho}(G)$ of a graph $G$ is the smallest integer $p$ such that $G$ is $(1,2,\dots ,p)$-packing colorable. Gastineau and Togni asked whether the subdivision $S(G)$ of every subcubic graph $G$ has $\chi_{\rho}(S(G))\leq 5$ and whether every subcubic graph, except the Petersen graph, is $(1,1,2,2)$-packing colorable; these questions were later conjectured by Bre\v{s}ar et al. Moreover,  Gastineau and Togni proved that a positive answer to the second question implies a positive answer to the first. In this paper, we completely resolve the second question for connected non-regular subcubic graphs, proving that they are $(1,1,2,2)$-packing colorable and hence satisfy $\chi_{\rho}(S(G)) \leq 5$. We also establish the same result for several classes of cubic graphs, including those with diamonds, certain cut-vertices, and bridges on short cycles.
Finally, we strengthen the recent result of Liu, Zhang, and Zhang [\textit{Discrete Math.} 348 (11) (2025). 114610] that every subcubic graph is $(1,1,2,2,3)$-packing colorable by proving that every connected cubic graph admits a $(1,1,2,2,k)$-packing coloring in which at most one vertex receives color $k$, where $k$ is arbitrary. This not only simplifies the existing argument but also strictly improves the bound.\end{abstract}


\section{Introduction}
All graphs considered in this paper are simple graphs. For a graph $G$, we denote its vertex set by $V(G)$ and its edge set by $E(G)$. For a vertex $x\in V(G)$, $N_G(x)$ represents the set of its neighbors. Besides, the degree of $x$, denoted by $d_G(x)$, is the number of the neighbors of $x$. We use $\Delta(G)$ and $\delta(G)$ to represent the maximum and minimum degree of $G$, respectively. A graph $G$ is called subcubic graph if $\Delta(G)\leq 3$, and it is called regular subcubic graph or simply cubic graph if $d_G(x)=3$ for every vertex $x$ in $G$. 

The distance between two vertices $x$ and $y$ in $G$ is denoted by $dist_G(x,y)$ or simply $dist(x,y)$. The square graph of $G$, $G^2$, is the graph obtained from $G$ after adding an edge between two vertices $x$ and $y$ in $G$ whenever  $dist_G(x,y)=2$. It is important to note that if $G^2$ is bipartite then $V(G)$ can be partitioned into two sets, $V_1$ and $V_2$, such that any two vertices in $V_i$, $1\leq i\leq 2$, are at distance at least three from each other.  
For $H \subseteq V (G)$, we denote by $G[H]$ the subgraph of $G$ induced by $H$. The subdivision of $G$, denoted by $S(G)$, is the graph obtained from $G$ after replacing each edge in $G$ by a path of length two. 

Given a sequence of positive integers $(s_1,s_2,\dots,s_k)$ where $s_1\leq s_2\leq \dots\leq s_k$, an $S$-packing coloring of a graph $G$ is a partition of the vertex set $V(G)$ into subsets $V_1,\; V_2,\dots, \; V_k$, such that for any two distinct vertices $x$ and $y$ in the same subset $V_i$, $1\leq i\leq k$, $dist_G(x,y)\geq s_i+1$. The packing chromatic number $\chi_{\rho}(G)$ of a graph $G$ is the smallest integer $k$ such that $G$ is $(1,2,\dots,k)$-packing colorable.  Under the name of broadcast chromatic number, this parameter was originally introduced by Goddard et al.~\cite{d}. Much of the research on the packing chromatic number has focused on subcubic graphs, and it has been studied in great depth, as detailed in the survey by Bre\v{s}ar et al.~\cite{bfk}.
Balogh, Kostochka, and Liu \cite{bkl}, as well as Bre\v{s}ar and Ferme \cite{bf}, demonstrated that the packing chromatic number for such graphs is unbounded. Several studies \cite{3,11,8,9,16,24,25} have explored methods for establishing bounds for $\chi_{\rho}(G)$ or identifying sequences $S$ that allow these graphs to be $S$-packing colorable. Additionally, Balogh, Kostochka, and Liu~\cite{3} showed $\chi_{\rho}(S(G))$ is bounded by 8 for subcubic graphs. Gastineau and Togni \cite{16} posed the question of whether $\chi_{\rho}(S(G))\leq 5$, which was later conjectured by Bre\v{s}ar et al.~\cite{9}.

Gastineau and Togni~\cite{16} further showed that for a subcubic graph $G$ to satisfy $\chi_{\rho}(S(G)) $ is bounded by $5$, it is sufficient for $G$ to be $(1, 1, 2, 2)$-packing colorable. Yang and Wu~\cite{a} proved that every 3-irregular subcubic graph is $(1,1,3)$-packing colorable and a simpler proof for the same result was provided in \cite{mai}. Furthermore, Bre\v{s}ar et al.~\cite{9} proved that if $G$ is a generalized prism of a cycle, then $G$ is $(1, 1, 2, 2)$-packing colorable if and only if $G$ is not the Petersen graph. Finally, Tarhini and Togni \cite{i} proved that every cubic Halin graph is $(1, 1, 2, 3)$-packing colorable. Besides, Liu et al.~\cite{m} proved that subcubic graphs with maximum average degree less than $\frac{30}{11}$ are $(1,1,2,2)$-packing colorable and Bre\v{s}ar, Kuenzel, and Rall \cite{BKR} proved the same result for claw-free subcubic graphs. Most recently, Liu, Zhang and Zhang \cite{LZZ24} established that every subcubic graph is $(1,1,2,2,3)$-packing colorable, thus showing that the packing chromatic number of the subdivision of any subcubic graph is at most $6$.

In~\cite{26,27,mai1}, Mortada and Togni investigated the $S$-packing coloring for special subclasses of subcubic graphs that are defined according to the existence and the number of the vertices with degree two among the neighbors of a vertex of degree three. Their proofs were based on the concept of a maximal independent set, where the weight of a vertex of degree three exceeds that of a vertex of degree two. Inspired by these results, we found that it is important to study the subcubic graph that contains at least one vertex of degree less than three. 

In this paper, we prove that every connected non-regular subcubic graph is $(1,1,2,2)$-packing colorable, which immediately implies $\chi_\rho(S(G)) \leq 5$ for this large class. We then extend the result to several important families of cubic graphs: those containing a diamond, those with a vertex whose removal yields three connected components, and those having a bridge whose ends lie on cycles of length at most four. As a conclusion, if $G$ is one of the previous graphs, then $\chi_{\rho}(S(G))\leq 5$. Finally, we strengthen the result of Liu, Zhang, and Zhang \cite{LZZ24} by proving that every connected cubic graph admits a $(1,1,2,2,k)$-packing coloring in which at most one vertex is colored $k$. This strictly improves the known bound, since $k$ can be chosen arbitrarily large, and gives a short and self-contained proof that $\chi_\rho(S(G)) \leq 6$ for all subcubic graphs.

\section{Non-Regular Subcubic Graphs} 
In this section, we deal with non-regular subcubic graphs.
\begin{theorem}
Every connected non-regular subcubic graph is $(1,1,2,2)$-packing colorable.
\end{theorem}
\begin{proof}
Suppose to the contrary that the statement is false and let $G$ be a counter-example with minimum order. We claim $\delta(G)=2$. In fact, otherwise, let $u\in G$ be a vertex of minimum degree and let $G'=G-u$. By the minimality of $G$, the graph $G'$ is $(1,1,2,2)$-packing colorable. Let $1_a$, $1_b$, $2_a$, $2_b$ be the colors used to color $G'$. Either $1_a$ or $1_b$ is not the color of the neighbor of $u$. By giving this color to $u$, we obtain a $(1,1,2,2)$-packing coloring of $G$, a contradiction.\\
For any vertex $x\in V(G)$, we define the following weight:
$$w(x)=1+\min \{dist(x,u): d_G(u)=2\}.$$
The defined weight satisfies the following properties:\begin{itemize}
\item[(P1)] If $x,y\in V(G)$ are adjacent vertices, then $|w(x)-w(y)|\leq 1$.
\item[(P2)] If $x\in V(G)$ is such that $d_G(x)=3$, then $w(x)=1+\min \{w(y): y\in N_G(x)\}$.
\end{itemize}
For any subgraph $K$ of $G$, we define its weight as follows:
$$w(K)=\displaystyle \sum_{x\in V(K)} w(x).$$
Let $S$ be a bipartite induced subgraph of $G$ such that $(|E(S)|,w(S))$ is maximum using lexicographic ordering. Let $\overline{S}$ be the subgraph induced by $V(G)\setminus S$ and consider a partition $\{S_1,S_2\}$ of $V(S)$ such that $S_1$ and $S_2$ are independent sets. Let $X_S$ be the subgraph of $G$ induced by the non-isolated vertices in $G[\overline{S}]$. Suppose $S$ is chosen such that $w(X_S)$ is maximum.\\
Our plan is to show that the graph $G^2[\overline{S}]$ is bipartite. Considering two independent sets in $G^2[\overline{S}]$, $H_1$ and $H_2$, forming a partition of $V(\overline{S})$, then coloring the vertices of $S_1,S_2,H_1,H_2$ by the colors $1_a,1_b,2_a,2_b$ respectively, leads to a $(1,1,2,2)$-packing coloring of $G$, which is a contradiction.\\
To prove that $G^2[\overline{S}]$ is bipartite, we will study the relations of the vertices of $V(\overline{S})$.
\begin{claim}\label{Claim1}
For all $x\in V(\overline{S})$, there exists a neighbor $x_1$ (resp., $x_2$) of $x$ in $S_1$ (resp., in $S_2$) such that $N_G(x_1)\cap S_2\neq\emptyset$ (resp., $N_G(x_2)\cap S_1\neq \emptyset$).
\end{claim}
\begin{proof}
If $x$ has no neighbors in $S_1$, then $S'=S\cup \{x\}$ is bipartite with $|E(S')|\geq |E(S)|$ and $w(S')>w(S)$, a contradiction. Thus, $x$ has neighbors in both $S_1$ and $S_2$. Suppose to the contrary that all neighbors of $x$ in $S_1$ have no neighbors in $S_2$. Let $S'_1=(S_1\setminus N_G(x))\cup \{x\}$ and $S'_2=S_2\cup (N_G(x)\cap S_1)$. The subgraph $S'=G[S'_1\cup S'_2]$ of $G$ is bipartite with $|E(S')|> |E(S)|$, a contradiction. Similarly for $S_2$.
\end{proof}\\
\\
By Claim \ref{Claim1}, we see that the connected components of $\overline{S}$ are vertices and edges. 
\begin{claim}\label{Claim2}
For all $z\in V(S)$, we have $|N_G(z)\cap V(\overline{S})|\leq 2$.
\end{claim}
\begin{proof}
Suppose to the contrary that there exists $z\in V(S)$ such that $|N_G(z)\cap V(\overline{S})|=3$. Obviously, $z$ is isolated in $S$. Then, by Claim \ref{Claim1}, any neighbor of $z$ in $V(\overline{S})$ has two neighbors in $S$ other than $z$. Hence, the neighbors of $z$ are $3$-vertices and any neighbor of these neighbors is in $V(S)$. Now, by Property (P2), $z$ has a neighbor $x\in V(\overline{S})$ such that $w(x)<w(z)$. Again, since $x$ is a $3$-vertex, by Property (P2), $x$ has a neighbor $y\in V(S)$ such that $w(y)<w(x)$. Without loss of generality, suppose that $y\in S_1$. Moreover, as $z$ is isolated in $S$, we may assume that $z\in S_2$, and so, by Claim \ref{Claim1}, $x$ has a neighbors in $S_2$ other than $z$. Thus, the set $(S_1\setminus \{y\})\cup \{x\}$ is independent. Therefore, the subgraph $S'=G[((S_1\setminus \{y\})\cup \{x\})\cup S_2]$ of $G$ is bipartite with $|E(S')|\geq |E(S)|$ and $w(S')>w(S)$, a contradiction. In fact, as $x$ and $y$ are adjacent, $y$ has at most two neighbors in $S$. On the other hand, $x$ has two neighbors in $V(S)$ other than $y$. So, when replacing $y$ by $x$ in $S$, the number of edges will remain the same or increase by one. Moreover, as $w(x)>w(y)$, we have $w(S')>w(S)$.
\end{proof}

\begin{claim}\label{Claim3}
Let $x\in V(\overline{S})$ such that $|N_G(x)\cap V(S)|=3$ and let $N_G(x)=\{x_1,x_2,x_3\}$ such that $x_1$ and $x_2$ are either both in $S_1$ or in $S_2$. Then, $d_S(x_3)=2$ and $w(x_3)\geq w(x)$.
\end{claim}
\begin{proof}
Without loss of generality, suppose that $x_1$ and $x_2$ are in $S_1$. Suppose to the contrary that $d_S(x_3)=1$. The bipartite induced subgraph $S'=G[S_1 \cup ((S_2\setminus \{x_3\})\cup \{x\})]$ of $G$ has $|E(S')|>|E(S)|$, a contradiction. We still need to prove that $w(x_3)\geq w(x)$. Again, for the sake of contradiction, suppose that $w(x_3)<w(x)$, then the bipartite induced subgraph $S'=G[S_1 \cup ((S_2\setminus \{x_3\})\cup \{x\})]$ of $G$ has $|E(S')|=|E(S)|$ and $w(S')>w(S)$, a contradiction.
\end{proof}\\
\\
Consequently, we have the following:
\begin{remark}\label{remark1}
Let $x,y,z\in V(\overline{S})$ such that $x$ and $y$ (resp., $y$ and $z$) have a common neighbor $u$ (resp., $v$) in $V(S)$. The following hold: \begin{itemize}
    \item[(i)] If $|N_G(y)\cap V(S)|=3$, then $u$ and $v$ are either both in $S_1$ or in $S_2$.
    \item[(ii)] If $|N_G(y)\cap V(S)|=2$, then either $u\in S_1$ and $v\in S_2$ or $u\in S_2$ and $v\in S_1$. 
\end{itemize}
\end{remark}
To complete the proof, we have to prove that $G^2[\overline{S}]$ is bipartite.
\begin{claim}\label{Claim4}
The graph $G^2[\overline{S}]$ is bipartite.
\end{claim}
\begin{proof}
Suppose to the contrary that $H$ is not bipartite. Then, $H$ contains an odd cycle $C=x_1\dots x_kx_1$. Let $I=\{1,\dots ,k\}$. We need to study two cases:\\
\\
\textbf{Case 1:} For all $i\in I$, we have $|N_G(x_i)\cap V(S)|=3$.\\
For notation convenience, we assume that the indices are taken cyclically, meaning that $x_0$ refers to $x_k$ and $x_{k+1}$ refers to $x_1$. Clearly, the distance in $G$ between $x_i$ and $x_{i+1}$ is two, for all $i\in I$.\\
Now, for $i\in I$, let $y_i\in S$ be a common neighbor of $x_i$ and $x_{i+1}$. By Remark \ref{remark1}, the vertices $y_1,\dots ,y_k$ are either all in $S_1$ or in $S_2$, say $S_1$. Moreover, by Claim \ref{Claim3}, for all $i\in I$, there exists $z_i\in S_2$ such that $w(z_i)\geq w(x_i)$. Thus, by Property (P2), we have the following:\begin{itemize}
\item[(C1)] For all $i\in I$, either $w(y_{i-1})<w(x_i)$ or $w(y_i)<w(x_i)$.
\end{itemize}
Let $I_r=\{i\in I: w(x_i)>w(y_i)\}$ and $I_l=\{i\in I: w(x_i)>w(y_{i-1})\}$. By Property (C1), we have $I_r\cup I_l = I$. Then, without loss of generality, we may suppose that $|I_r|\geq \frac{|I|}{2}=\frac{k}{2}$. But $k$ is odd, then $|I_r|\geq \frac{|I|+1}{2}>|I\setminus I_r|$. Then,
\begin{equation}\label{Eq1}
\begin{aligned}
\sum_{i\in I} (w(x_i)-w(y_i)) &= \sum_{i\in I_r} (w(x_i)-w(y_i)) + \sum_{i\in I\setminus I_r} (w(x_i)-w(y_i))\\
&\geq \sum_{i\in I_r} (1) + \sum_{i\in I\setminus I_r} (-1)\\
&= |I_r| - |I\setminus I_r|\\
&>0.
\end{aligned}
\end{equation}
Now, let $S'_1=(S_1\setminus \{y_1,\dots ,y_k\})\cup \{x_1,\dots ,x_k\}$ and let $S'=G[S'_1\cup S_2]$. The subgraph $S'$ (see Figure \ref{Fig1}) of $G$ is bipartite with $|E(S')|\geq |E(S)|$ and $w(S')>w(S)$, a contradiction. In fact, $|E(S')|\geq |E(S)|$ since, for all $i\in I$, $x_i$ has a neighbor in $S_2$ while $y_i$ has at most one neighbor in $S_2$. Moreover, $w(S')>w(S)$ follows from (\ref{Eq1}).
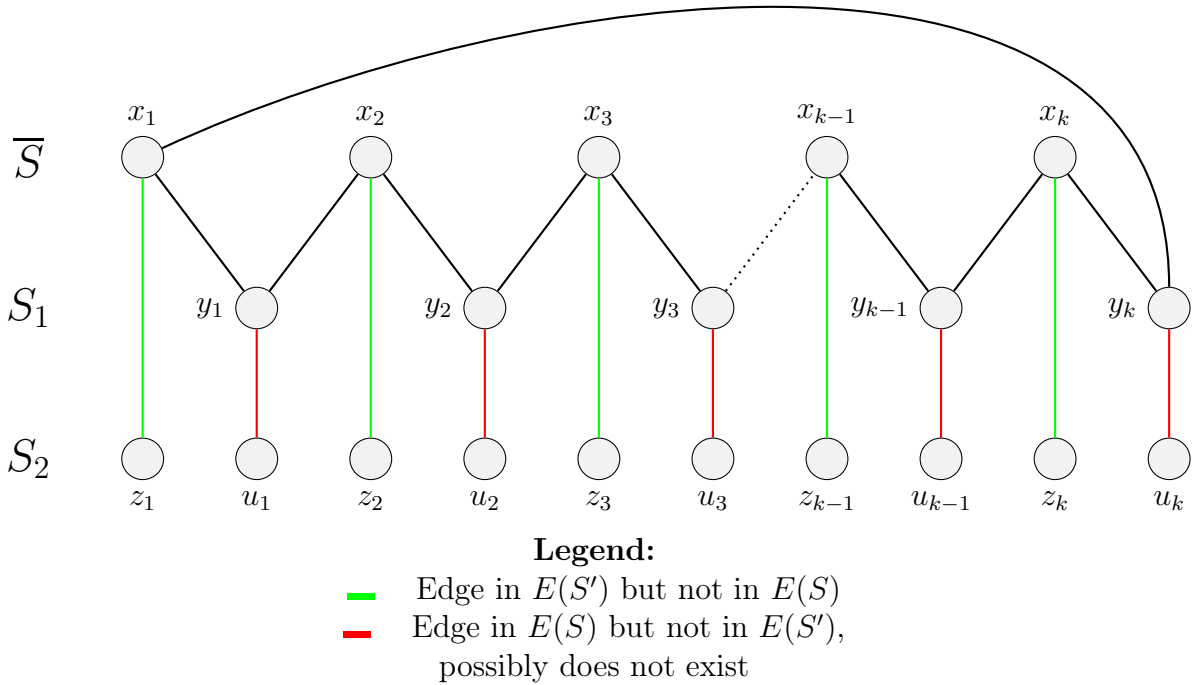
\begin{figure}[h!]
\centering
\begin{tikzpicture}[
    node/.style={circle, draw, fill=gray!10, minimum size=5.5mm},
    edge/.style={thick},
    green edge/.style={thick, draw=green},
    red edge/.style={thick, draw=red},
    black edge/.style={thick, draw=black},
    dotted edge/.style={thick, dotted},
    weight/.style={font=\bfseries\small}
]

% Nodes
\node[node,label={above:$x_1$}] (x1) at (0, 0) {};
\node[node,label={above:$x_2$}] (x2) at (3, 0) {};
\node[node,label={above:$x_3$}] (x3) at (6, 0) {};
\node[node,label={above:$x_{k-1}$}] (xk-1) at (9, 0) {};
\node[node,label={above:$x_k$}] (xk) at (12, 0) {};

\node[node,label={left:$y_1$}] (y1) at (1.5, -2) {};
\node[node,label={left:$y_2$}] (y2) at (4.5, -2) {};
\node[node,label={left:$y_3$}] (y3) at (7.5, -2) {};
\node[node,label={left:$y_{k-1}$}] (yk-1) at (10.5, -2) {};
\node[node,label={left:$y_k$}] (yk) at (13.5, -2) {};

\node[node,label={below:$u_1$}] (u1) at (1.5, -4) {};
\node[node,label={below:$u_2$}] (u2) at (4.5, -4) {};
\node[node,label={below:$u_3$}] (u3) at (7.5, -4) {};
\node[node,label={below:$u_{k-1}$}] (uk-1) at (10.5, -4) {};
\node[node,label={below:$u_k$}] (uk) at (13.5, -4) {};

\node[node,label={below:$z_1$}] (z1) at (0, -4) {};
\node[node,label={below:$z_2$}] (z2) at (3, -4) {};
\node[node,label={below:$z_3$}] (z3) at (6, -4) {};
\node[node,label={below:$z_{k-1}$}] (zk-1) at (9, -4) {};
\node[node,label={below:$z_k$}] (zk) at (12, -4) {};

% Green edges between xi and zi
\draw[green edge] (x1) -- (z1);
\draw[green edge] (x2) -- (z2);
\draw[green edge] (x3) -- (z3);
\draw[green edge] (xk-1) -- (zk-1);
\draw[green edge] (xk) -- (zk);

% Red edges between yi and wi
\draw[red edge] (y1) -- (u1);
\draw[red edge] (y2) -- (u2);
\draw[red edge] (y3) -- (u3);
\draw[red edge] (yk-1) -- (uk-1);
\draw[red edge] (yk) -- (uk);

% Black edges between xi and yi
\draw[black edge] (x1) -- (y1);
\draw[black edge] (x2) -- (y2);
\draw[black edge] (x3) -- (y3);
\draw[black edge] (xk-1) -- (yk-1);
\draw[black edge] (xk) -- (yk);

% Black edges between xi and yi-1
\draw[black edge] (x2) -- (y1);
\draw[black edge] (x3) -- (y2);
\draw[black edge] (xk) -- (yk-1);
\draw[black edge] (x1) to[out=25, in=90] (yk);

% Dotted lines indicating continuation
\draw[dotted edge] (y3) -- (xk-1);

% Labels for sets
\node at (-1.5, 0) {\Large $\overline{S}$}; % Label for set S̅
\node at (-1.5, -2) {\Large $S_1$}; % Label for set S1
\node at (-1.5, -4) {\Large $S_2$}; % Label for set S2

\end{tikzpicture}

% Adding the legend
\begin{minipage}[t]{0.5\textwidth}
\centering
\textbf{Legend:}\\
\textcolor{green}{\rule{10pt}{2pt}} \quad Edge in $E(S')$ but not in $E(S)$ \\
\textcolor{red}{\rule{10pt}{2pt}} \quad Edge in $E(S)$ but not in $E(S')$, possibly does not exist
\end{minipage}
\caption{Representation of Case 1 of Claim \ref{Claim4}.}
\end{figure}\label{Fig1}\\
\\
\textbf{Case 2:} There exists $i\in I$ such that $|N_G(x_i)\cap V(S)|=2$.\\
Without loss of generality, suppose that $|N_G(x_1)\cap V(S)|=2$. By Remark \ref{remark1}, we can deduce that there exists $m\in \{2,\dots ,k\}$ such that $|N_G(x_m)\cap V(S)|=2$. Suppose that $m$ is chosen minimum. That is, $|N_G(x_i)\cap V(S)|=3$ for all $i\in \{2,\dots ,m-1\}$. Let $y_i\in S$ be the common neighbor of $x_i$ and $x_{i+1}$, for all $i\in \{1,\dots ,m-1\}$. By Remark \ref{remark1}, all vertices $y_1,\dots ,y_{m-1}$ are either all in $S_1$ or in $S_2$, say $S_1$. Let $z_i$ be the neighbor of $x_i$ in $S_2$, for all $i\in \{1,\dots ,m\}$.\\
\\
\textbf{Subcase 2.1:} $x_1$ and $x_m$ are not adjacent.\\
Let $S'_1=(S_1\setminus \{y_1,\dots ,y_{m-1}\})\cup \{x_1,\dots ,x_m\}$ and let $S'=G[S'_1\cup S_2]$. The subgraph $S'$ (see Figure \ref{Fig2}) of $G$ is bipartite with $|E(S')|>|E(S)|$, a contradiction.\\
\\
\textbf{Subcase 2.2:} $x_1$ and $x_m$ are adjacent.\\
Without loss of generality, suppose that $w(x_1)\leq w(x_m)$. First, suppose that $w(y_1)<w(x_1)$. Let $S'_1=(S_1\setminus \{y_1\})\cup \{x_1\}$ and $S'=G[S'_1\cup S_2]$. The subgraph $S'$ of $G$ is bipartite with $|E(S')|=|E(S)|$ and $w(S')>w(S)$, a contradiction. So, we may assume that $w(y_1)\geq w(x_1)$. Thus, $w(z_1)<w(x_1)$. Let $S'_1=(S_1\setminus \{y_{m-1}\})\cup \{x_m\}$ and $S'_2=(S_2\setminus \{z_1\})\cup \{x_1\}$. In addition, let $S'=G[S'_1\cup S'_2]$. Clearly, $|E(S')|\geq |E(S)|$. Indeed, we removed at most two edges incident to $z_1$ and one edge incident to $y_{m-1}$, while we added the edges $x_1y_1,x_1x_m,x_mz_m$. Hence, we may assume that $|E(S')|=|E(S)|$ and $z_1$ has two neighbors in $S_1$. If $w(S')>w(S)$, we obtain a contradiction. Then, $w(S')\leq w(S)$. Now, as $w(z_1)<w(x_1)$, we must have $w(y_{m-1})>w(x_m)$. Precisely, $w(x_1)=w(z_1)+1$ and $w(x_m)=w(y_{m-1})-1$. That is, $w(S')=w(S)$. Now, $X_{S'}=(X_S\setminus \{x_1,x_m\})\cup \{x_{m-1},y_{m-1}\}$. Note that $w(x_{m-1})\geq w(y_{m-1})-1=w(x_m)\geq w(x_1)$. Then, $w(X_{S'})=w(X_S)+(w(x_{m-1})-w(x_1))+(w(y_{m-1})-w(x_m))\geq w(X_S)+0+1>w(X_S)$, a contradiction.
\end{proof}
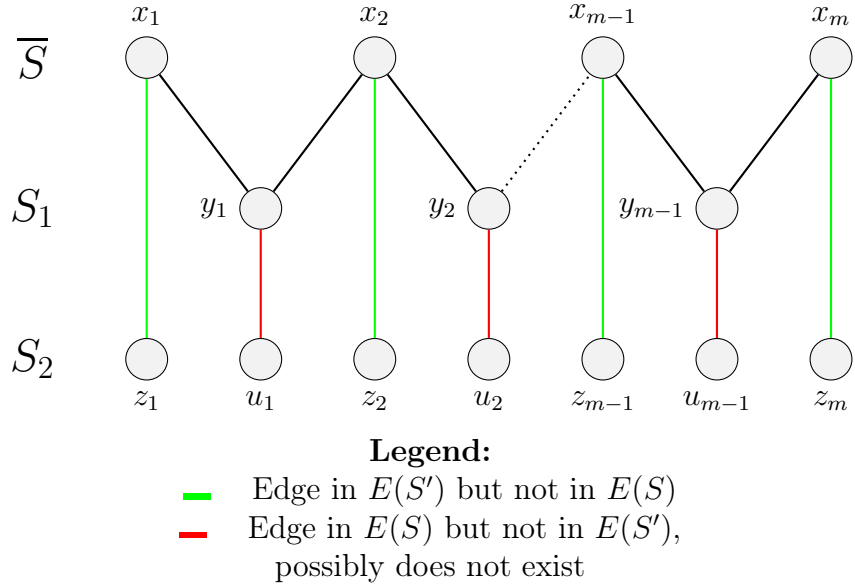
\begin{figure}[h!]
\centering
\begin{tikzpicture}[
    node/.style={circle, draw, fill=gray!10, minimum size=5.5mm},
    edge/.style={thick},
    green edge/.style={thick, draw=green},
    red edge/.style={thick, draw=red},
    black edge/.style={thick, draw=black},
    dotted edge/.style={thick, dotted},
    weight/.style={font=\bfseries\small}
]

% Nodes
\node[node,label={above:$x_1$}] (x1) at (0, 0) {};
\node[node,label={above:$x_2$}] (x2) at (3, 0) {};
\node[node,label={above:$x_{m-1}$}] (xm-1) at (6, 0) {};
\node[node,label={above:$x_m$}] (xm) at (9, 0) {};

\node[node,label={left:$y_1$}] (y1) at (1.5, -2) {};
\node[node,label={left:$y_2$}] (y2) at (4.5, -2) {};
\node[node,label={left:$y_{m-1}$}] (ym-1) at (7.5, -2) {};

\node[node,label={below:$u_1$}] (u1) at (1.5, -4) {};
\node[node,label={below:$u_2$}] (u2) at (4.5, -4) {};
\node[node,label={below:$u_{m-1}$}] (um-1) at (7.5, -4) {};

\node[node,label={below:$z_1$}] (z1) at (0, -4) {};
\node[node,label={below:$z_2$}] (z2) at (3, -4) {};
\node[node,label={below:$z_{m-1}$}] (zm-1) at (6, -4) {};
\node[node,label={below:$z_m$}] (zm) at (9, -4) {};

% Green edges between xi and zi
\draw[green edge] (x1) -- (z1);
\draw[green edge] (x2) -- (z2);
\draw[green edge] (xm-1) -- (zm-1);
\draw[green edge] (xm) -- (zm);

% Red edges between yi and wi
\draw[red edge] (y1) -- (u1);
\draw[red edge] (y2) -- (u2);
\draw[red edge] (ym-1) -- (um-1);

% Black edges between xi and yi
\draw[black edge] (x1) -- (y1);
\draw[black edge] (x2) -- (y2);
\draw[black edge] (xm-1) -- (ym-1);

% Black edges between xi and yi-1
\draw[black edge] (x2) -- (y1);
\draw[black edge] (xm) -- (ym-1);

% Dotted lines indicating continuation
\draw[dotted edge] (y2) -- (xm-1);

% Labels for sets
\node at (-1.5, 0) {\Large $\overline{S}$}; % Label for set S̅
\node at (-1.5, -2) {\Large $S_1$}; % Label for set S1
\node at (-1.5, -4) {\Large $S_2$}; % Label for set S2

\end{tikzpicture}
% Adding the legend
\begin{minipage}[t]{0.5\textwidth}
\centering
\textbf{Legend:}\\
\textcolor{green}{\rule{10pt}{2pt}} \quad Edge in $E(S')$ but not in $E(S)$ \\
\textcolor{red}{\rule{10pt}{2pt}} \quad Edge in $E(S)$ but not in $E(S')$, possibly does not exist
\end{minipage}
\caption{Representation of Subcase 2.1 of Claim \ref{Claim4}. Note that $x_1$ and $x_m$ may have a neighbor in $\overline{S}$.}
\end{figure}\label{Fig2}\\
\\
Now, let $\{H_1,H_2\}$ be a partition of $G^2[\overline{S}]$ such that $H_1$ and $H_2$ are independent in $G^2[\overline{S}]$. By coloring the vertices of $H_1$ (resp., $H_2$) by $2_a$ (resp., $2_b$) and the vertices of $S_1$ (resp., $S_2$) by $1_a$ (resp., $1_b$), we obtain a $(1,1,2,2)$-packing coloring of $G$, which is a contradiction.
\end{proof}

\section{Cubic Graphs}
Cubic graphs form the most delicate subfamily of subcubic graphs with respect to $(1,1,2,2)$-packing colorability, since all vertices have maximum degree and the minimum-degree argument used in Section~2 no longer applies. While Theorem~1 completely resolves the problem for non-regular subcubic graphs, the case of cubic graphs requires a separate and more refined approach.

In this section, we identify several broad structural configurations that guarantee $(1,1,2,2)$-packing colorability for cubic graphs. In particular, we prove that every connected cubic graph containing a diamond (Corollary~1), every connected cubic graph with a vertex whose removal yields three connected components 
(Corollary~2), and every connected cubic graph with a bridge whose ends lie on cycles of length at most four (Corollary~3) are $(1,1,2,2)$-packing colorable. 
Together with Theorem~1, these results imply $\chi_\rho(S(G)) \leq 5$ for a wide class of connected subcubic graphs, leaving only very restricted cubic graphs as potential counterexamples to the conjecture of Gastineau and Togni.

Finally, we present Corollary~4, which provides a short and elementary argument showing that every connected cubic graph admits a $(1,1,2,2,k)$-packing coloring in which at most one vertex is assigned color $k$, where $k$ is arbitrary. This result strengthens the recent theorem of Liu, Zhang, and Zhang~\cite{LZZ24}, who proved that every subcubic graph is $(1,1,2,2,3)$-packing colorable. Our approach shows that the extension to a fifth color is immediate once a $(1,1,2,2)$-coloring is known for $G-e$ for any edge $e$ of $G$, and that the additional color can be restricted to a single vertex. This provides a conceptually simpler proof of $\chi_\rho(S(G)) \leq 6$ for all subcubic graphs and narrows the gap between $(1,1,2,2)$-colorability and $(1,1,2,2,k)$-colorability much more than previously recognized.
\begin{corollary}\label{corollary 1}
    Every connected cubic graph containing a diamond is $(1,1,2,2)$-packing colorable.
\end{corollary}
\begin{proof}
Let $G$ be a cubic graph that contains a diamond and let $x,y,u,v$ be the vertices of the diamond such that $N(x)=\{y,u,v\}$ and $N(y)=\{x,u,v\}$. Without loss of generality, we may suppose that $G$ is connected. If $u$ is adjacent to $v$, then the result is trivial. So, we may assume that $uv\notin E(G)$. Let $u_1$ (resp., $v_1$) be the neighbor of $u$ (resp., $v$) other than $x$ and $y$. Let $G'$ be the graph obtained by deleting $x$ from $G$ and adding the edge $uv$. Clearly, $G'$ is non-regular. Consider a $(1,1,2,2)$-packing coloring of $G'$, using the colors $1_a,1_b,2_a,2_b$. If neither of the neighbors of $x$ is colored by $1_a$ (resp., $1_b$), then by coloring $x$ using $1_a$ (resp., $1_b$) we obtain a $(1,1,2,2)$-packing coloring of $G$. So, we may assume that at least one of the neighbors of $x$ is colored by $1_a$ (resp., $1_b$). Without loss of generality, we may suppose that $u$ is colored by $1_a$. We distinguish two cases.\\
\\
\textbf{Case 1:} $v$ is colored by $1_b$.\\
Here, $y$ is not colored by $1_a$ nor $1_b$. If $u_1$ is not colored by $1_b$, then we can recolor $u$ by $1_b$ and $x$ by $1_a$. The result follows. So, we may assume that $u_1$ (resp., $v_1$) is colored by $1_b$ (resp., $1_a$). Note that the vertices that are at a distance of at most $2$ from $x$ are $u,u_1,v,v_1,y$. Hence, either $2_a$ or $2_b$ can be assigned to $x$ to obtain a $(1,1,2,2)$-packing coloring of $G$.\\
\\
\textbf{Case 2:} $y$ is colored by $1_b$.\\
Let $k$ be the color of $v$. Obviously, $k\notin \{1_a,1_b\}$. If $v_1$ is colored by $1_b$, then recolor $v$ by $1_a$ and $x$ by $k$. The result follows. So, we may assume that $v_1$ is colored by $k'\neq 1_b$. Suppose that $u_1$ is colored by $1_b$. Then, we may assume that $\{k,k'\}=\{2_a,2_b\}$, since otherwise one can find a color $2_i$ for $x$ to obtain a $(1,1,2,2)$-packing coloring of $G$. Here, recolor $v$ by $1_a$ and $x$ by $k$. The result follows. Finally, we may suppose that $u_1$ is colored by $k''\neq 1_b$. Recolor $u$ and $v$ by $1_b$, $y$ by $1_a$, and $x$ by $k$. The result follows.
\end{proof}

\begin{corollary}\label{corollary 2}
    Every connected cubic graph $G$ having a vertex $v$ such that $G-v$ has three connected components is $(1,1,2,2)$-packing colorable.
\end{corollary}
\begin{proof}
Let $G$ be a connected cubic graph and $v$ be a vertex in $G$ such that $G-v$ has three connected components $G_1,G_2,G_3$. Clearly, $G_1,G_2,G_3$ are non-regular. Let $v_1,v_2,v_3$ be the neighbors of $v$ in $G_1,G_2,G_3$, respectively. We can consider a $(1,1,2,2)$-packing coloring of $G_i$, for all $i\in \{1,2,3\}$, such that $v_i$ is not colored by $1_a$. Now, color $v$ by $1_a$ to obtain a $(1,1,2,2)$-packing coloring of $G$.
\end{proof}

\begin{corollary}
Let $G$ be a connected cubic graph that contains a bridge $e = xy$ such that each of $x$ and $y$ belongs to a cycle of length at most four. 
Then, $G$ admits a $(1,1,2,2)$-packing coloring.
\end{corollary}
\begin{proof}
Let $G' = G - e$. By Theorem~1, $G'$ admits a $(1,1,2,2)$-packing coloring. Fix such a coloring $c : V(G') \to \{1_a,1_b,2_a,2_b\}$. If $c$ is still valid after restoring $e$, we are done. Otherwise, the only possible conflicts in $G$ created by restoring $xy$ are:
\begin{enumerate}
  \item[\textup{(A)}] $x$ and $y$ receive the same color;
  \item[\textup{(B)}] $x$ and a neighbor of $y$ have the same $2$-color;
  \item[\textup{(C)}] $y$ and a neighbor of $x$ have the same $2$-color.
\end{enumerate}
 Let $y_1$ and $y_2$ be the neighbors of $y$ other than $x$. Here is the treatment of each conflict when it appears:\begin{itemize}
 \item \textbf{Resolution of conflict (A).}
 In this case, it is enough to switch the colors $1_a$ and $1_b$ (resp. $2_a$ and $2_b$) if $x$ and $y$ have the same 1-color (resp. 2-color) in the connected component containing $x$ in $G'$. Hence, conflict (A) is always resolvable.
 
\item \textbf{Resolution of conflict (B) when conflict (C) doesn't occur.} 
 Suppose $x$ and $y_1$ have the same 2-color, say $2_a$. If $y_2$ is not colored by $2_b$, then switch $2_a$ and $2_b$ in the  component of $x$ in $G'$ and so the conflict is removed. Else; that is $y_2$ is of color $2_b$. In this case, a global switch in the component of $x$ is not enough to resolve the conflict. Because $y$ lies on a cycle $C$ of length $3$ or $4$, we will show that we can locally recolor to make one of $y_1,y_2$ a $1$-colored vertex and then resolve the conflict accordingly.

If $G[\{y,y_1,y_2\}]$ is a triangle, then recolor $y_1$ by $1_i$, $i\in\{a,b\}$, such that $y_i$ has no neighbor of color $1_i$ in $G\setminus \{y\}$ and recolor $y$ by $1_j$ such that $j\in \{a,b\}\setminus\{i\}$. This removes the conflict. 

For the case when  $G[\{y,y_1,y_2\}]$ is not a triangle, then $y$, $y_1$ and $y_2$ are contained in a  square. Let $z$ be the fourth vertex of this square. If there exists $i\in\{a,b\}$ such that none of the neighbors of $y_1$ in $G\setminus \{y\}$ is of color $1_i$ then proceed as in the previous case; that is  recolor $y_1$ by $1_i$ and recolor $y$ by $1_j$ such that $j\in \{a,b\}\setminus\{i\}$. Else, if there exists $i\in\{a,b\}$ such that none of the neighbors of $y_2$ in $G\setminus \{y\}$ is of color $1_i$, then recolor $y_2$ by $1_i$,   recolor $y$ by $1_j$ such that $j\in \{a,b\}\setminus\{1_i\}$, and switch the colors $2_a$ and $2_b$ in the  component of $x$ in $G'$. Else, both colors $1_a$ and $1_b$ are given to the neighbors of $y_1$ (resp. $y_2$) in $G\setminus \{y\}$. Let $z_1$ (resp. $z_2$) be the third neighbor of $y_1$ (resp. $y_2$). Hence, $z$ is of color $1_l$ and $z_1$ and $z_2$ are both of color $1_k$ such that $\{1_l,1_k\}=\{1_a,1_b\}$. We have the following observation:
\begin{claim}
    The vertex $z$ can be given a color other than $1_l$ and then $y_1$ or $y_2$ can be recolored in order to remove the conflict.
\end{claim}
\begin{proof}
If $z$ has no neighbor of color $1_k$,  recolor $z$ by $1_k$.  Then,  none of the neighbors of $y_1$ in $G\setminus \{y\}$ is now of color $1_l$ and so the conflict can be removed as before by recoloring $y_1$ by $1_l$ and $y$ by $1_k$. Else, let $z'$ be the neighbor of $z$ of color $1_k$. If $z'$ has no neighbor other than $z$ of color $1_l$, then recolor $z'$ by $1_l$, $z$ by $1_k$, $y_1$ by $1_l$, and $y$ by $1_k$. Hence, the conflict is removed.  Otherwise, $z'$ has two neighbors of color $1_l$ and so, since   $G[\{y,y_1,y_2,z\}]$ is a square, then there exists $i\in \{a,b\}$ such that $z$ is at distance at least three from each vertex of color $2_i$ in $G\setminus \{y_1,y_2\}$. If $i=a$, recolor $z$ by $2_a$, $y_1$ by $1_l$, $y$ by $1_k$, and so the conflict is removed. Else; that is $i=b$, then recolor $z$ by $2_b$, $y_2$  by $1_l$, $y$ by $1_k$, and finally switch the two colors $2_a$ and $2_b$ in the connected component containing $x$. 
\end{proof}

\item \textbf{Resolution when both (B) and (C) occur.}
In this case, there exists $i\in \{a,b\}$ (resp. $j\in\{a,b\}$) such that $x$ (resp $y$) has no neighbor of color $1_i$ (resp $1_j$). Thus, recolor $x$ by $1_i$. For $y$, recolor $y$ by $1_j$ if $i\neq j$. Else; that is $i=j$, switch $1_a$ and $1_b$ in the connected component containing $y$ in $G'$ and then recolor $y$ by $1_l$ such that $l\neq i$.\end{itemize}
\end{proof}

\begin{corollary}
Every connected cubic graph is $(1,1,2,2,k)$-packing colorable in such a way that at most one vertex receives the color $k$.
\end{corollary}
\begin{proof}
Let $G$ be a connected cubic graph and let $e=xy\in E(G)$. Set $G'=G-e$. By Theorem~1, $G'$ admits a $(1,1,2,2)$-packing coloring; fix one such coloring
$c:V(G')\to\{1_a,1_b,2_a,2_b\}$.
If $c$ remains a $(1,1,2,2)$-packing coloring after adding back the edge $xy$, we are done. Otherwise, the same conflicts as in the previous proof will appear.

If \textup{(A)} appears, recolor $x$ by $k$; this yields a $(1,1,2,2,k)$-packing coloring with a unique vertex colored $k$.  
If exactly one of \textup{(B)} or \textup{(C)} appears, recolor the endpoint involved in that conflict (\emph{i.e.}, $x$ in case \textup{(B)}, or $y$ in case \textup{(C)}) by $k$; again we use $k$ at a single vertex.

It remains to handle the case when both \textup{(B)} and \textup{(C)} appear. In this case, there exists $i\in \{a,b\}$ such that none of the neighbors of $x$ is of color $1_i$.  Recolor $x$ by $1_i$ and recolor $y$ by $k$ and so we obtain
a $(1,1,2,2,k)$-packing coloring of $G$ with at most one vertex colored $k$.
\end{proof}

\section{Conclusions and Open Problems}

\begin{corollary}
    Let $G$ be a connected subcubic graph that contains one of the following:\begin{itemize}
        \item[(i)] A 1-vertex or a 2-vertex,
      \item[(ii)] A diamond, 
       \item[(iii)] A vertex $v$ such that $G-v$ has three connected components, or
       \item[(iv)] A bridge whose each of its ends relies on a cycle of length at most four.
    \end{itemize}
   Then, $\chi_{\rho}(S(G))\leq 5$.
\end{corollary}
The above results significantly narrow the scope of the conjecture of Bre\v{s}ar et al., reducing it to a very restricted class of cubic graphs. In particular, it is natural to ask the following question:
\begin{problem}
Is every connected cubic graph with a cut-vertex $(1,1,2,2)$-packing colorable?
\end{problem}

A positive answer would settle the $(1,1,2,2)$-colorability for all cubic graphs that are not $3$-connected, leaving only the $3$-connected cubic graphs as potential counterexamples to the conjecture of Bre\v{s}ar et al.

\end{document}